\documentclass{amsart}

\usepackage{url,hyperref}
\usepackage[vcentermath]{youngtab}
\newtheorem{theorem}{Theorem}
\newtheorem*{theorem*}{Theorem}
\newtheorem{proposition}[theorem]{Proposition}

\newtheorem{corollary}[theorem]{Corollary}
\newtheorem*{corollary*}{Corollary}

\theoremstyle{definition}

\def\Fl{\mathcal{F}\ell}

\def\GL{\operatorname{GL}}

\def\Sym{\operatorname{Sym}}

\def\CC{\mathbf{C}}

\def\S{\mathfrak{S}}

 \author{Andrew Berget}
\email{berget@math.ucdavis.edu}
\address{Department of Mathematics\\
  University of California\\
  Davis, CA 95616\\
  USA} \date{\today} \title[Equality of symmetrized tensors and the
flag variety]{Equality of symmetrized tensors\\{\tiny and}\\the
  coordinate ring of the flag variety} \thanks{The author is partially
  supported by NSF grant DMS-0636297. The author thanks Victor Reiner,
  who advised the thesis where this proof first appeared.}
\begin{document}
\maketitle
\section{Introduction}
The goal of this note is to give a transparent proof of a result
of da Cruz and Dias da Silva on the equality of symmetrized
decomposable tensors. This will be done by explaining that their
result follows from the fact that the coordinate ring of a flag
variety is a unique factorization domain.

Let $\lambda$ be a partition of a positive integer $n$ and let
$\chi^\lambda$ be the irreducible character of the symmetric group
$\S_n$ corresponding to $\lambda$. There is a right action of $\S_n$
on $V^{\otimes n}$, where $V$ is a finite-dimensional complex vector
space, by permuting tensor positions. Let $T_\lambda$ be the
endomorphism of $V^{\otimes n}$ given by
\[
(v_1 \otimes \dots \otimes v_n)
T_\lambda=\frac{\chi^\lambda(1)}{n!}\sum_{\sigma \in \S_n}
\chi^\lambda(\sigma) v_{\sigma(1)} \otimes \dots \otimes
v_{\sigma(n)}.
\]

The question of when there is an equality,
\begin{align}\label{eq:equality}
(v_1 \otimes \dots \otimes v_n)
T_\lambda
=
(u_1 \otimes \dots \otimes u_n)
T_\lambda  
\end{align}
was studied by Merris \cite{merris}, and Chollet and Marcus
\cite{marcus1,marcus2} in the mid-1970s. Since then, many papers were
written on this question and partial results were given by a variety
of authors over the years (for example,
\cite{marcus3,dds3,dds4,fernandez,fonseca}). However, the question was
fully resolved only recently in two papers by da Cruz and Dias da
Silva \cite{dds2,dds1}. 

A tableau of shape $\lambda$ will be a filling of the numbers
$1,2,\dots,n$ into the boxes of the Young diagram of $\lambda$. A
tableaux is said to be standard if the numbers in both its rows and
columns increase.
\begin{theorem*}[da Cruz--Dias da Silva]
  There is an equality of symmetrized decomposable tensors as in
  \eqref{eq:equality} if and only if the following conditions are
  satisfied:
  \begin{enumerate}
  \item Every tableau of shape $\lambda$ whose columns index linearly
    independent subsets of $(v_1,\dots,v_n)$ also index linearly
    independent subsets of $(u_1,\dots,u_n)$.
  \item If the columns of a tableau of shape $\lambda$ index
    independent subsets of $(v_1,\dots,v_n)$, and $C_j$ denotes the
    numbers in column $j$ of this tableaux, then there is some
    permutation $\sigma$ such that for all $j$, the span of those
    $\{v_i: i \in C_j\}$ is equal to the span $\{u_i: i \in
    C_{\sigma(j)}\}$.
  \item In the situation above, the product the determinants relating
    $\{v_i: i \in C_j\}$ to $\{u_i: i \in C_{\sigma(j)}\}$ is $1$.
  \end{enumerate}
\end{theorem*}
The idea of the proof, and the structure of the paper, is as follows. We
first reduce to the case of applying Young symmetrizers to the
decomposable tensors. Then we use the fact that applying a Young
symmetrizer to a tensor can be viewed as multiplication in the
coordinate ring of a flag variety. The proof of the theorem follows by
by interpreting it as a statement about unique factorization in this
ring. 

The proof is straightforward, once a person becomes familiar with the
various guises of the representation theory of the general linear
group. Perhaps the myriad of ways of understanding these
representations explains why the result withstood understanding for so
long. Indeed, the description of the coordinate ring of a flag variety
used here goes back to Deruyts, and a modern treatment of it was known
to, among others, Towber \cite{towber} in the 1970s. That this
coordinate ring is a unique factorization domain is what makes the
theorem non-trivial. It is a vestige of the first fundamental theorem
of invariant theory, as explained in \cite[Chapter 9]{fulton}.  

A certain amount of familiarity the representation theory of the
general linear group will be assumed. We point the reader to Fulton's
book \cite{fulton} for a beautiful synthesis of all the relevant
ideas, giving explicit pointers to particular results as we use them.

In point of notation, we will write $v^\otimes$ for the tensor product
$v_1 \otimes v_2 \otimes \dots \otimes v_n$. 

All the results in this paper hold not just over $\CC$, but over an
arbitrary field of characteristic zero.
\section{Reduction}
Let $T$ be a tableau of shape $\lambda$, $a_T$ its row symmetrizer,
and $b_T$ its column antisymmetrizer. These are given by
\[
\sum_{\sigma \in \operatorname{Row}(T)} \sigma, \quad \sum_{\sigma \in
  \operatorname{Col}(T)} \operatorname{sign}(\sigma)\sigma,
\]
respectively. Here the sums are over the row and column groups of $T$,
which are the subgroups of $\S_n$ that stablize each row and each
column of $T$, respectively. For example, using cycle notation for
permutations in $\S_n$, if
\[
T = \young(234,15) 
\]
then $b_T = (1-(12))(1-(35))$ while
\[
a_T = \left(1+(23)+(24)+(34)+(234)+(243)\right)(1+(15)).
\]
A product $c_T := b_Ta_T$ is called a \textit{Young symmetrizer} and
the right ideal in $\CC\S_n$ generated by a Young symmetrizer is an
irreducible $\CC\S_n$-module with character $\chi^\lambda$
\cite[Chapter 7]{fulton} while the image of $b_T a_T$ on $V^{\otimes
  n}$ is zero, or irreducible for the diagonal action of the general
linear group with highest weight $\lambda$ \cite[Proposition
8.1]{fulton}. It is clear that $v^\otimes b_T$ is not zero if and only
if the sets of vectors indexed by the columns of the tableau $T$ are
linearly independent.

\begin{proposition}\label{prop:red}
  There is an equality $v^\otimes T_\lambda = u^\otimes T_\lambda$ if
  and only if for all tableau $T$ of shape $\lambda$ there is an
  equality $v^\otimes c_T = u^\otimes c_T$.
\end{proposition}
\begin{proof}
  Recall that $T_\lambda$ is the projector of $\CC\S_n$ to its
  $\chi^\lambda$ isotypic component.  Now,
  \[
  v^\otimes T_\lambda = u^\otimes T_\lambda \implies
  v^\otimes T_\lambda c_T = u^\otimes T_\lambda c_T \implies
  v^\otimes  c_T = u^\otimes  c_T.
  \]
  The first implication is trivial and the second follows from the
  fact that $T_\lambda$ fixes any Young symmetrizer $c_T$ in $\CC\S_n$
  if $T$ has shape $\lambda$.

  For the other implication, recall the fact that the sum $\sum_T c_T$
  over all tableaux of shape $\lambda$ is a non-zero scalar multiple
  of $T_\lambda$\footnote{This follows from the fact $\sum_T c_T$ is
    not zero in $\CC\S_n$, central, idempotent, and annhilates Young
    symmetrizers of different shapes.}. Hence, if $v^\otimes c_T =
  u^\otimes c_T$ for all $T$ of shape $\lambda$ then,
  \[
  \sum_T v^\otimes c_T = \sum_T u^\otimes c_T \implies v^\otimes
  T_\lambda = u^\otimes T_\lambda.\qedhere
  \]
\end{proof}

\section{Flag varieties and the shape algebra}
Denote the dimension of $V$ by $r$. Let $\Fl_{r}$ be the complex
variety of complete flags in $V$:
\[
\Fl_{r} = \{ 0 \subset V_1 \subset V_2 \subset \dots \subset V_{r-1}
\subset V : \dim V_i = i, i \leq i \leq r\}.
\]
The coordinate ring of $\Fl_{r}$, sometimes called the shape algebra
\cite{towber}, can be written as a quotient
\[
A = A(\Fl_r):=\Sym \left(\textstyle \bigwedge (\CC^r) \right)/ Q,
\]
where $\Sym$ and $\bigwedge$ denote the symmetric and exterior algebra
functors, and $Q$ is an ideal of quadratic relations whose precise
definition is not needed (see \cite[Proposition~9.1]{fulton} for
this). We can write this symmetric algebra of $\bigwedge \CC^r$ as,
\[
\Sym \left(\textstyle \bigwedge (\CC^r) \right) =
\bigoplus_{a_1,a_2,\dots,a_r \geq 0} \Sym^{a_r}(\textstyle \bigwedge^r
\CC^r) \otimes \dots \otimes \Sym^{a_1}(\textstyle \bigwedge^1 \CC^r),
\]
and the ideal $Q$ is homogeneous with respect to the obvious
$\mathbb{N}^r$-grading. The graded piece of the quotient $A$ that is
indexed by $(a_1,a_2,\dots,a_r)$ can, instead, be indexed by the
partition that has $a_i$ columns of length $i$, for
$i=1,2,\dots,r$. We write the graded piece of $A$ corresponding to
$\lambda$ as $A_\lambda$. In this way we see that the coordinate ring
of $\Fl_r$ can be written as a direct sum $A = \bigoplus A_\lambda$,
the sum over partitions $\lambda$ with at most $r$ parts.

If we have a partition $\lambda$ with $a_i$ columns of length $i$ for
$1 \leq i \leq r$, then it is a fact that $A_\lambda$ is the
irreducible representation of $\GL(V)$ with highest weight $\lambda$
\cite[Theorem 8.2]{fulton}. Further, for all tableau $T$ of shape
$\lambda$ there is a unique isomorphism of $\GL_r(V)$ representations
\cite[Proposition 8.1]{fulton},
\[
A_\lambda \approx V^{\otimes n} c_T.
\]
Let $T$ denote the \textit{column super standard} tableau of shape
$\lambda$, whose entries are $1,2,\dots,n$ when read top-to-bottom,
left-to-right. By the universal property of $A_\lambda$ \cite[Theorem
8.1]{fulton} note that the composition of multiplication maps,
\begin{multline*}
  (\textstyle \bigwedge^r \CC^r)^{\otimes a_r} \otimes \dots \otimes
  (\textstyle \bigwedge^1 \CC^r)^{\otimes a_1} \to
  \Sym^{a_r}(\textstyle \bigwedge^r \CC^r) \otimes \dots \otimes
  \Sym^{a_1}(\textstyle \bigwedge^1 \CC^r)\\ \to A_\lambda \to
  V^{\otimes n} c_T,
\end{multline*}
takes an element,
\[
(v_1 \wedge \dots \wedge v_r) \otimes \dots \otimes (v_{(r-1)a_r}
\wedge \dots \wedge v_{ra_r}) \otimes \dots \otimes v_{n-a_1} \otimes
\dots \otimes v_n,
\]
to $(v_1 \otimes v_2 \otimes \dots \otimes v_n) c_T$.
\section{Proofs}
We can now, in one fell swoop, prove two results of Dias da Silva and
Fonseca (unpublished), as well as the theorem of da Cruz and Dias da
Silva.
\begin{theorem}[Dias da Silva--Fonseca]
  The tensor $v^\otimes c_T$ is not zero if and only if $v^\otimes
  b_T$ is not zero if and only if the columns of $T$ index linearly
  independent subsets of $(v_1,\dots,v_n)$.
\end{theorem}
\begin{proof}
  Let $C_j$ denote the set of numbers in the $j$th column of $T$. Then
  $v^\otimes c_T$ is not zero if and only if the product of the elements
  \[
  \bigwedge_{i \in C_1} v_i, \bigwedge_{i \in C_2} v_i, \bigwedge_{i
    \in C_3} v_i,\dots
  \]
  is not zero in $A$. However, $A$ is an integral domain
  \cite[Proposition 8.2]{fulton} hence the product is not zero if and
  only if each of its factors is not zero. The latter is true if and
  only if each set of vectors $\{v_i: i \in C_j\}$ is linearly
  independent.
\end{proof}
Combining this result with Proposition~\ref{prop:red} we have given
another proof of Gamas's theorem on the vanishing of symmetrized
decomposable tensors (\textit{cf.}, \cite{berget}).
\begin{corollary}[Gamas]
  The symmetrized tensor $v^\otimes T_\lambda$ is not zero if and only
  there is a tableau of shape $\lambda$ whose columns index linearly
  independent sets.
\end{corollary}
The following slightly stronger result was first due to Dias da Silva
and Fonseca (unpublished).
\begin{corollary}[Dias da Silva--Fonseca]
  The symmetrized tensor $v^\otimes T_\lambda$ is not zero if and only
  there is a standard tableau of shape $\lambda$ whose columns index
  linearly independent sets.
\end{corollary}
Recall that a tableau is standard if the numbers in each row and
column increase.
\begin{proof}
  Suppose that $v^\otimes c_T$ is not zero, but $T$ is not a standard
  tableau. It is well known that there is an equality of the form $c_T
  = \sum_j c_{S} x_S$, where the sum is over standard tableaux of
  shape $\lambda$ and $x_S \in \CC\S_n$. The result follows.
\end{proof}
We now give the promised proof of the theorem on equality of
symmetrized decomposable tensors.
\begin{theorem}[da Cruz--Dias da Silva]
  There is an equality of symmetrized decomposable tensors as in
  \eqref{eq:equality} if and only if the following conditions are
  satisfied:
  \begin{enumerate}
  \item Every tableau of shape $\lambda$ whose columns index a
    linearly independent subsets of $(v_1,\dots,v_n)$ also index
    linearly independent subsets of $(u_1,\dots,u_n)$.
  \item If the columns of a tableau of shape $\lambda$ index
    independent subsets of $(v_1,\dots,v_n)$ and $C_j$ denotes the
    numbers in column $j$ of this tableaux, then there is some
    permutation $\sigma$ such that for all $j$,
    \[
    \bigwedge_{i \in C_j} v_i = c_j\bigwedge_{i \in C_{\sigma(j)}} u_i
    \]
    for some non-zero scalar $c_j$. That is, for all $j$, the span of
    those $v_i$ indexed by $C_j$ is equal to the span of those $u_i$
    indexed by $C_{\sigma(j)}$.
  \item In the situation above, the product of all the $c_j$s is $1$.
  \end{enumerate}
\end{theorem}
\begin{proof}
  After relabeling the vectors, it is sufficient to check that
  $v^\otimes c_T = u^\otimes c_T$ for $T$ the column super standard
  tableau of shape $\lambda$. Interpreting this statement in the
  coordinate ring $A$, this is stating that the product in $A$ of the
  wedges
  \[
  \bigwedge_{i \in C_1} v_i, \bigwedge_{i \in C_2} v_i, \bigwedge_{i \in C_3} v_i,\dots,
  \]
  is equal to the product of the wedges
  \[
  \bigwedge_{i \in C_1} u_i, \bigwedge_{i \in C_2} u_i, \bigwedge_{i
    \in C_3} v_i, \dots.
  \]
  However, $A$ is a unique factorization domain
  \cite[Section~9.2]{fulton} and the equality of the products is
  exactly conditions (1), (2) and (3).
\end{proof}


\begin{thebibliography}{70}
\bibitem{berget}
A. Berget, 
A short proof of Gamas's theorem.
\textit{Linear Algebra Appl.} \textbf{430} (2009), no. 2-3, 791--794. 

\bibitem{marcus1}
J. Chollet, M. Marcus,
Decomposable symmetrized tensors.
\textit{Linear and Multilinear Algebra} \textbf{6} (1978/79), no. 4, 317--326.

\bibitem{marcus3}
J. Chollet, M. Marcus,
On the equality of decomposable symmetrized tensors.  
\textit{Linear and Multilinear Algebra} \textbf{13} (1983), no. 3, 253--266.

\bibitem{dds2}
H. da Cruz, J. A. Dias da Silva,
Equality of immanantal decomposable tensors. II. 
\textit{Linear Algebra Appl.} \textbf{395} (2005), 95--119. 


\bibitem{dds1}
H. da Cruz, J. A. Dias da Silva, 
Equality of immanantal decomposable tensors. 
\textit{Linear Algebra Appl.} \textbf{401} (2005), 29--46.


\bibitem{dds3} J. A. Dias da Silva, Flags and equality of tensors.
  \textit{Linear Algebra Appl.}  \textbf{232} (1996), 55--75.

\bibitem{dds4} J. A. Dias da Silva, Colorings and equality of tensors.
  \textit{Linear Algebra Appl.}  \textbf{342} (2002), 79--91.


\bibitem{fernandez}
M. Fernandes,
Pairs of matrices that have the same immanant. 
\textit{Linear and Multilinear Algebra} \textbf{40} (1996), no. 3, 193--201. 

\bibitem{fonseca}
A. Fonseca, On the equality of families of decomposable symmetrized
tensors.  \textit{Linear Algebra Appl.}  \textbf{293} (1999), no. 1-3, 1--14.

\bibitem{fulton}
W. Fulton, 
Young tableaux. 
With applications to representation theory and geometry. \textit{London Mathematical Society Student Texts}, \textbf{35}. Cambridge University Press, Cambridge, 1997.

\bibitem{marcus2}
M. Marcus,
Decomposable symmetrized tensors and an extended LR decomposition theorem.
\textit{Linear and Multilinear Algebra} \textbf{6} (1978/79), no. 4, 327--330. 

\bibitem{merris}
R. Merris, Equality of decomposable symmetrized tensors. 
\textit{Canad. J. Math.} \textbf{27} (1975), 1022-1024.

\bibitem{towber}
J. Towber,
Two new functors from modules to algebras.
\textit{J. Algebra} \textbf{47} (1977), no. 1, 80--104. 

 \end{thebibliography}
\end{document}